\documentclass{article}
\usepackage{amsfonts}
\usepackage{latexsym}
\usepackage{mathrsfs}
\usepackage{amssymb}
\usepackage{amsmath}
\usepackage{amsthm}
\usepackage{indentfirst}
\usepackage{color}
\usepackage{float}
\usepackage{graphicx}

\allowdisplaybreaks
\newtheorem{theorem}{\color{black}\indent Theorem}[section]
\newtheorem{lemma}{\color{black}\indent Lemma}[section]

\newtheorem{remark}{\color{black}\indent Remark}[section]

\begin{document}
\title{\LARGE\bf Lifespan of solutions to a damped fourth-order wave equation with logarithmic nonlinearity}
\author{Yuzhu Han$^{\dag}$ \qquad Qi Li}
 \date{}
 \maketitle

\footnotetext{\hspace{-1.9mm}$^\dag$Corresponding author.\\
Email addresses: yzhan@jlu.edu.cn(Y. Han).

\thanks{
$^*$Supported by NSFC (11401252) and by The Education Department of Jilin Province (JJKH20190018KJ).}}
\begin{center}
{\noindent\it\small School of Mathematics, Jilin University,
 Changchun 130012, P.R. China}
\end{center}

\date{}
\maketitle

{\bf Abstract}\ This paper is devoted to the lifespan of solutions to a damped fourth-order wave equation
with logarithmic nonlinearity
$$u_{tt}+\Delta^2u-\Delta u-\omega\Delta u_t+\alpha(t)u_t=|u|^{p-2}u\ln|u|.$$
Finite time blow-up criteria for solutions at both lower and high initial energy levels are established,
and an upper bound for the blow-up time is given for each case.
Moreover, by constructing a new auxiliary functional and making full use of the strong damping term,
a lower bound for the blow-up time is also derived.



\section{Introduction}
\setcounter{equation}{0}

In this paper, we are concerned with the following initial boundary value
problem for a damped fourth-order wave equation with logarithmic nonlinearity
\begin{equation}\label{1.1}
\begin{cases}
u_{tt}+\Delta^2u-\Delta u-\omega\Delta u_t+\alpha(t)u_t=|u|^{p-2}u\ln|u|, & (x,t)\in\Omega\times(0,T),\\
u(x,t)=\Delta u(x,t)=0, & (x,t)\in\partial \Omega\times(0,T),\\
u(x,0)=u_0(x), \quad u_t(x,0)=u_1(x), &x\in \Omega,
\end{cases}
\end{equation}
where $\Omega\subset \mathbb{R}^n(n\geq1)$ is a bounded domain with smooth boundary $\partial \Omega$,
$T\in(0,+\infty]$ is the maximal existence time of the solution $u(x,t)$, $\omega>0$,
$\alpha(t): [0,\infty)\rightarrow [0,\infty)$ is a nonincreasing bounded differentiable function,
and the exponent $p$ satisfies
$$(A)\qquad 2<p<2_*,$$
where $2_*=+\infty$ if $n\leq 4$ and $2_*=\dfrac{2n}{n-4}$ if $n\geq 5$.

Problems like \eqref{1.1} have their roots in many branches of physics such as nuclear physics, optics and geophysics.
They may also be used to describe some phenomena of granular materials such as the longitudinal motion
of an elastic-plastic bar. Interested reader may refer to \cite{Al,An-1,An-2,An-3,Guo-Li}
for more background of problems like \eqref{1.1}.
It is well known that the damping terms (both strong $\Delta u_t$ and weak $u_t$)
prevent solutions from blowing up while the nonlinear terms force solutions to blow up.
So it is of great interest to investigate how one dominates the other,
and much effort has been devoted to this direction during the past few years.
For example, Gazzola et al. \cite{Gazzola} investigated the following damped wave equation
\begin{equation}\label{sec}
 u_{tt}-\Delta u-\omega\Delta u_t+\mu u_t=|u|^{p-2}u
\end{equation}
in a bounded domain of $\mathbb{R}^n$, where $\omega\geq 0$, $\mu>-\omega\lambda_1$ and $p>2$.
By using the potential well method first proposed by Sattinger et al. \cite{Payne,Sattinger},
they obtained the existence of global and finite time blow-up solutions to \eqref{sec} for
initial data at different energy levels. As for the damped fourth-order wave equations,
Lin et al. \cite{Lin} considered the following hyperbolic equation with strong damping
\begin{equation}\label{fou-1}
 u_{tt}+\Delta^2u-\Delta u-\omega\Delta u_t=f(u)
\end{equation}
in a bounded domain of $\mathbb{R}^n$ with $\omega>0$.
Under certain conditions on the initial data and on the nonlinearity $f$,
they proved the existence of global weak solutions and global strong solutions
by using the classical potential well method. When the nonlinearity $f(u)$ grows super-linearly with respect to $u$ as $u$ tends to infinity,
the solutions to \eqref{fou-1} may blow up in finite time.
In 2018, Wu \cite{Wu} considered the following initial boundary value problem
\begin{equation}\label{fou-2}
\begin{cases}
u_{tt}+\Delta^2u-\Delta u-\omega\Delta u_t+\alpha(t)u_t=|u|^{p-2}u, & (x,t)\in\Omega\times(0,T),\\
u(x,t)=\Delta u(x,t)=0, & (x,t)\in\partial \Omega\times(0,T),\\
u(x,0)=u_0(x), \quad u_t(x,0)=u_1(x), &x\in \Omega,
\end{cases}
\end{equation}
where $\omega$ and $\alpha(t)$ fulfill the same conditions as that of problem \eqref{1.1} and
$p$ satisfies the so-called subcritical condition, i.e.,
\begin{equation*}
 p\in(2,\infty)\quad if\quad  n\leq 4;\quad p\in\big(2,\frac{2n-4}{n-4}\big)\quad  if \quad n\geq 5.
\end{equation*}
After showing that the unstable set is invariant under the flow of \eqref{fou-2},
he proved a blow-up result for problem \eqref{fou-2} with initial energy smaller than the depth of the
potential well, by applying concavity argument. Moreover, a lower bound for the blow-up time is derived.
Later, problem \eqref{fou-2} was reconsidered by Guo et al. \cite{Guo-Li} and
the results of \cite{Wu} were extended in two aspects. The first is that they obtained a blow-up result
for high initial energy, and the second is that lower bound for the blow-up time is also derived
for some supercritical $p$, with the help of inverse H\"{o}lder's inequality and interpolation inequality.

On the other hand, evolution equations with logarithmic nonlinearity have
also attracted more and more attention in recent years,
due to their wide applications to quantum field theory and other applied sciences.
Among the huge amount of interesting literature,
we only refer the interested reader to \cite{Cao1,Chen1,Chen2,Di,Han,Han-Cao-Sun,Ji1,Le1,Le2,Lian,Ma},
where qualitative properties of solutions to hyperbolic or parabolic equations with logarithmic nonlinearities were studied.
In particular, Di et al. \cite{Di} considered the following initial boundary value problem for a semilinear wave equation
with strong damping and logarithmic nonlinearity
\begin{equation}\label{sec-2}
\begin{cases}
u_{tt}-\Delta u-\Delta u_t=|u|^{p-2}u, & (x,t)\in\Omega\times(0,T),\\
u(x,t)=0, & (x,t)\in\partial \Omega\times(0,T),\\
u(x,0)=u_0(x), \quad u_t(x,0)=u_1(x), &x\in \Omega,
\end{cases}
\end{equation}
when $\Omega\subset \mathbb{R}^n(n\geq1)$ is a bounded domain with smooth boundary $\partial \Omega$,
$2<p<+\infty$ if $n=1,2$ and $2<p<\frac{2n}{n-2}$ if $n\geq 3$.
The existence of global or finite time blow-up solutions to problem \eqref{sec-2} with initial energy less
than or equal to the depth of the potential well was investigated by using the potential well method.
Moreover, the decay rate of the energy functional was obtained for global solutions
and upper and lower bounds for the blow-up time were also derived for blow-up solutions.
However, the case that the initial energy is larger than the depth of the potential well was not considered in \cite{Di},
and we do not know whether or not problem \eqref{sec-2} admits finite time blow-up solutions for this case.
In addition, the lower bound for the blow-up time was obtained when $p$ is subcritical,
i.e., $p<\frac{2n-2}{n-2}$. When $p\in[\frac{2n-2}{n-2},\frac{2n}{n-2})$ for $n\geq 3$,
whether a lower bound for the blow-up time can be obtained is still open.

Motivated mainly by \cite{Di,Guo-Li,Wu}, we will consider problem \eqref{1.1} and investigate how the damping terms and logarithmic
nonlinearity determine the blow-up conditions and blow-up time of the solutions.
More precisely, we shall present some sufficient conditions for the solutions to problem \eqref{1.1} to blow up in
finite time with both lower and high initial energy and derive an upper bound for the blow-up time for each case.
Moreover, we also estimate a lower bound for the blow-up time, which, thanks to the strong damping term, also includes some supercritical case.
For simplicity, we only consider \eqref{1.1} for the case $\omega=1$ and $\alpha(t)\equiv 1$.
The main results can be extended to the general case with little difficulty.

The organization of this paper is as follows. In Section 2, as preliminaries,
some notations, definitions and lemmas that will be used in the sequel are introduced.
Finite time blow-up of solutions and upper bound for the blow-up time with lower and
high initial energy will be considered in Section 3 and Section 4, respectively.
In Section 5 we derive a lower bound for the blow-up time.

\section{Preliminaries}
\setcounter{equation}{0}

In this section, we introduce some notations and lemmas which will be used in the sequel.
In what follows, we denote by $\|\cdot\|_r$ the $L^r(\Omega)$-norm ($1\leq r\leq\infty$),
by $(\cdot ,\cdot )$ the $L^2(\Omega)$-inner product
and by $\lambda_1>0$ the first eigenvalue of $-\Delta$ in $\Omega$ under homogeneous Dirichlet boundary condition.
 Set
$$H=\{u\in H^2(\Omega)\cap H_0^1(\Omega):\ u=\Delta u=0\  \text{on}\  \partial\Omega\},$$
and equip it with the norm
$$\|u\|_H=\sqrt{\|\Delta u\|_2^2+\|\nabla u\|_2^2}.$$
For simplicity, we also denote the $H^1(\Omega)$-norm by
$$\|u\|=\sqrt{\|u\|_2^2+\|\nabla u\|_2^2}.$$
Obviously, for any $u\in H$, we have
\begin{equation}\label{lam}
 \lambda_1\|u\|^2\leq \|u\|_H^2.
\end{equation}
For any $u\in H$, define
\begin{equation}\label{j}
J(u)=\dfrac{1}{2}\|u\|_
  H^2-\dfrac{1}{p}\int_\Omega|u|^p\ln|u|{\rm d}x+\dfrac{1}{p^2} \|u\|_p^p,
\end{equation}
\begin{equation}\label{i}
I(u)=\|u\|_H^2-\int_\Omega|u|^p\ln|u|{\rm d}x,
\end{equation}
\begin{equation}\label{N}
\mathcal{N}=\{u\in H\setminus\{0\}: I(u)=0\},
\end{equation}
\begin{equation}\label{d}
d=\inf\limits_{u\in H\setminus\{0\}}\sup\limits_{\lambda>0}J(\lambda u)=\inf\limits_{u\in \mathcal{N}}J(u),
\end{equation}
where $\mathcal{N}$ is called the Nehari manifold and $d$ is the depth of the potential well (also called mountain pass level). In what follows, we shall show that $\mathcal{N}$ is non-empty and $d$ is positive.

The following lemma gives some properties of the so-called fibering map $J(\lambda u)$.
Since the proof is more or less standard (see \cite{Di} for example), we omit it here.

\begin{lemma}\label{fibering}
Let $p$ satisfy (A). Then for any $u\in H\setminus\{0\}$, we have
\begin{eqnarray*}
\mathrm{(i)} \ \lim_{\lambda\rightarrow0^+}J(\lambda u)=0, \ \lim_{\lambda\rightarrow+\infty}J(\lambda u)=-\infty. \ \ \ \ \ \ \ \ \ \ \ \ \ \ \ \ \ \ \ \ \ \ \ \ \ \ \ \ \ \ \ \ \ \ \ \ \ \ \ \ \ \ \ \ \ \ \ \ \ \
\end{eqnarray*}

$\mathrm{(ii)}$ \ there exists a unique $\lambda^*=\lambda^*(u)>0$ such that $\frac{d}{d\lambda}J(\lambda u)|_{\lambda=\lambda^*}=0$.
$J(\lambda u)$ is increasing on
$0<\lambda<\lambda^*$, decreasing on $\lambda^*<\lambda<+\infty$ and takes its maximum at $\lambda=\lambda^*$.

$\mathrm{(iii)}$ $I(\lambda u)>0$ on $0<\lambda<\lambda^*$, $I(\lambda u)<0$ on $\lambda^*<\lambda<+\infty$ and $I(\lambda^* u)=0$.
\end{lemma}

Let $\sigma$ be any positive number such that $p+\sigma<2_*$.
Then it is well known that the embedding from $H$ to $L^{p+\sigma}(\Omega)$ is compact and
there is a positive constant $B_\sigma$ such that
\begin{equation}\label{embedding}
  \|u\|_{p+\sigma}\leq B_\sigma\|u\|_H,\qquad \forall\ u\in H.
\end{equation}

\begin{lemma}\label{c0}
There is a positive constant $C_*$ such that $\|u\|_H\geq C_*$ for any $u\in \mathcal{N}$.
\end{lemma}

\begin{proof}
First, it follows from Lemma \ref{fibering} (iii) that $\mathcal{N}$ is non-empty.
For any $u\in \mathcal{N}$, using \eqref{embedding} and the basic inequality $\ln s\leq\dfrac{1}{e\sigma}s^\sigma$
for $s\geq 1$ and $\sigma>0$, we have
  \begin{equation}\label{2.6}
   \begin{split}
     \|u\|_H^2=&\int_\Omega|u|^p\ln|u|{\rm d}x=\int_{\Omega_1}|u|^p\ln|u|{\rm d}x+\int_{\Omega_2}|u|^p\ln|u|{\rm d}x\\
     \leq&\int_{\Omega_2}|u|^p\ln|u|{\rm d}x\leq \dfrac{1}{e\sigma}\int_{\Omega_2}|u|^{p+\sigma}{\rm d}x\\
     \leq&\dfrac{1}{e\sigma}\|u\|_{p+\sigma}^{p+\sigma}\leq \dfrac{B^{p+\sigma}_\sigma}{e\sigma}\|u\|_{H}^{p+\sigma},
   \end{split}
  \end{equation}
where $\Omega_1=\{x\in\Omega: |u(x)|<1\}$ and $\Omega_2=\{x\in\Omega: |u(x)|\geq1\}$.
Recalling that $p>2$, we obtain from \eqref{2.6} that $\|u\|_H\geq\Big(
\dfrac{e\sigma}{B^{p+\sigma}_\sigma}\Big)^{1/(p+\sigma-2)}\triangleq C_*$.
The proof is complete.
\end{proof}

\begin{lemma}\label{depth}
  The depth $d$ of the potential well is positive and
  there is a nonnegative function $v_0\in \mathcal{N}$ such that $J(v_0)=d$.
\end{lemma}

\begin{proof}
  By \eqref{j} and \eqref{i} we have
  \begin{equation}\label{i-j}
    J(u)=\dfrac{p-2}{2p}\|u\|^2_H+\dfrac{1}{p}I(u)+\dfrac{1}{p^2}\|u\|_p^p, \quad\ u\in H.
  \end{equation}
Therefore, for any $u\in \mathcal{N}$, by combining Lemma \ref{N} with \eqref{i-j} we obtain
 \begin{equation}\label{d0}
 J(u)\geq \dfrac{p-2}{2p}\|u\|^2_H\geq \dfrac{p-2}{2p}C_*^2 \triangleq d_0>0.
\end{equation}
By the definition of $d$ one sees that $d\geq d_0$, i.e., $d$ is positive.

To show that $d$ can be attained,
let $\{v_k\}_{k=1}^\infty\subset\mathcal{N}$ be a minimizing sequence of $J$.
It is easy to check that $\{|v_k|\}_{k=1}^\infty\subset\mathcal{N}$ is also a minimizing sequence of $J$.
Therefore, without loss of generality, we may assume that $v_k\geq 0$ a.e. in $\Omega$ for all $k\in \mathbb{N}$.
Then $J(v_k)$ is bounded, which, together with \eqref{i-j}, implies
that $\{v_k\}_{k=1}^\infty\subset\mathcal{N}$ is bounded in $H$.
Noticing that the embedding from $H$ to $L^{p+\sigma}(\Omega)$ is compact,
we see that there is a subsequence of $\{v_k\}_{k=1}^\infty\subset\mathcal{N}$,
which we still denote by $\{v_k\}_{k=1}^\infty\subset\mathcal{N}$, and a $v_0\in H$ such that
\begin{equation*}
\begin{split}
v_k\rightharpoonup v_0 \ & weakly \ in \ H \ as\ k \rightarrow\infty,\\
v_k\rightarrow v_0 \ & \ strongly \ in \ L^{p+\sigma}(\Omega) \ as\ k \rightarrow\infty,\\
v_k\rightarrow v_0 \ & \ a.e. \ in \ \Omega \ as\ k \rightarrow\infty.
\end{split}
\end{equation*}
Hence, $v_0\geq 0$ a.e. in $\Omega$. Furthermore, using the dominated convergence theorem, we obtain
\begin{equation}\label{ln}
\int_\Omega |v_0|^p\ln|v_0|{\rm d}x=\lim\limits_{k\rightarrow \infty}\int_\Omega |v_k|^p\ln|v_k|{\rm d}x,
\end{equation}
\begin{equation}\label{vp}
\int_\Omega |v_0|^p{\rm d}x=\lim\limits_{k\rightarrow \infty}\int_\Omega |v_k|^p{\rm d}x.
\end{equation}
Moreover, by the weak lower semicontinuity of $\|\cdot\|_H$, we have
\begin{equation}\label{h}
\|v_0\|_H\leq\liminf\limits_{k\rightarrow\infty}\|v_k\|_H.
\end{equation}
Therefore, it follows from \eqref{ln}-\eqref{h} that
\begin{equation}\label{jv0}
J(v_0)\leq\liminf\limits_{k\rightarrow\infty}J(v_k)=d,
\end{equation}
and
\begin{equation}\label{iv0}
I(v_0)\leq\liminf\limits_{k\rightarrow\infty}I(v_k)=0.
\end{equation}

It remains to show that $v_0\not\equiv 0$ and $I(v_0)=0$ to complete the proof.
By \eqref{ln} and Lemma \ref{c0} we know
\begin{equation*}
 \begin{split}
    \int_\Omega |v_0|^p\ln|v_0|{\rm d}x=&\lim\limits_{k\rightarrow \infty}\int_\Omega |v_k|^p\ln|v_k|{\rm d}x\\
      =&\lim\limits_{k\rightarrow \infty}\|v_k\|_H^2\geq C_*^2,
 \end{split}
\end{equation*}
which implies that $v_0\not\equiv 0$.

If $I(v_0)<0$, then by Lemma \ref{fibering} (iii) we know that there exists a $\lambda^*\in(0,1)$ such that
$I(\lambda^*v_0)=0$, i.e., $\lambda^*v_0\in\mathcal{N}$. By the definition of $d$, we see
\begin{equation*}
  \begin{split}
     d\leq&J(\lambda^*v_0)=\dfrac{(p-2)\lambda^{*2}}{2p}\|v_0\|_H^2+\dfrac{\lambda^{*p}}{p^2}\|v_0\|_p^p \\
       =&\lambda^{*2}\Big[\dfrac{p-2}{2p}\|v_0\|_H^2+\dfrac{\lambda^{*p-2}}{p^2}\|v_0\|_p^p\Big]\\
       <&\lambda^{*2}\Big[\dfrac{p-2}{2p}\|v_0\|_H^2+\dfrac{1}{p^2}\|v_0\|_p^p\Big]\\
       \leq&\lambda^{*2}\liminf\limits_{k\rightarrow\infty}\Big[\dfrac{p-2}{2p}\|v_k\|_H^2+\dfrac{1}{p^2}\|v_k\|_p^p\Big]\\
       =&\lambda^{*2}\liminf\limits_{k\rightarrow\infty}J(v_k)=\lambda^{*2}d,
  \end{split}
\end{equation*}
a contradiction. Therefore, $I(v_0)=0$ and $v_0\not\equiv 0$, which means that $v_0\in\mathcal{N}$.
Recalling \eqref{jv0} and the definition of $d$ again one sees that $J(v_0)=d$.
The proof is complete.
\end{proof}

In this paper, we consider weak solutions to problem \eqref{1.1}.
For completeness, we state, without proof, the local existence theorem
which can be established by slightly modifying the argument in \cite{Liu-Xu}.
Sometimes $u(x,t)$ will be simply written as $u(t)$ if no confusion arises.

\begin{theorem}\label{local-solution}
(\cite{Guo-Li,Wu}) Let $u_0\in H$ and $u_1\in L^2(\Omega)$. Then the problem \eqref{1.1} admits a unique weak solution
$u\in L^\infty(0,T_0;H)$, $u_t\in L^\infty(0,T_0;L^2(\Omega))\cap L^2(0,T_0;H_0^1(\Omega))$, for $T_0>0$ suitably small.
Moreover, the energy functional satisfies
\begin{equation}\label{ei}
 E'(t)=-\int_\Omega (u_t^2+|\nabla u_t|^2){\rm d}x\leq 0,
\end{equation}
where
\begin{equation}\label{e}
 E(t)=E(u(t))=\dfrac{1}{2}\|u_t\|_2^2+J(u(t)).
\end{equation}
\end{theorem}

At the end of this section, we present the well-known concavity lemma
which will play essential role in proving the blow-up result.

\begin{lemma}\label{concave}(See \cite{Han2018,Levine})
Suppose that a positive, twice-differentiable function $\psi(t)$ satisfies the inequality
$$\psi''(t)\psi(t)-(1+\theta)(\psi'(t))^2\geq0,$$
where $\theta>0$. If $\psi(0)>0$, $\psi'(0)>0$, then $\psi(t)\rightarrow\infty$ as $t\rightarrow t_*\leq t^*=\frac{\psi(0)}{\theta\psi'(0)}$.
\end{lemma}

\section{Blow-up for lower initial energy}
\setcounter{equation}{0}

In this section, we will investigate the blow-up phenomena of solutions to problem \eqref{1.1}
with lower initial energy. We first show that the unstable set $\mathcal{U}$ is invariant under the
flow of problem \eqref{1.1}, where
\begin{equation}\label{u}
\mathcal{U}=\{u\in H: I(u)<0,\ J(u)<d\},
\end{equation}
and $d$ is the depth of the potential well defined in \eqref{d}.

\begin{lemma}\label{inv}
  Let $u_0\in \mathcal{U}$ and $u_1\in L^2(\Omega)$ such that $E(0)<d$.
  Then $u(t)\in \mathcal{U}$ for all $t\in[0,T)$ and
  \begin{equation}\label{lower-bound-in-u}
   \dfrac{p-2}{2p}\|u(t)\|_H^2+\dfrac{1}{p^2}\|u(t)\|_p^p>d,\quad\forall\ t\in[0,T).
  \end{equation}
\end{lemma}

\begin{proof}
First, it follows from \eqref{j}, \eqref{e} and \eqref{ei} that
$$J(u(t))\leq E(t)\leq E(0)<d,\quad\forall\  t\in[0,T).$$
Therefore, in order to prove $u(t)\in \mathcal{U}$ for all $t\in[0,T)$,
it suffices to show that $I(u(t))<0$ for all $t\in[0,T)$.
Assume by contradiction that there exists a $t_1\in(0,T)$ such that $u(t_1)\in \mathcal{N}$.
Then by the variational definition of $d$, we obtain
$$d\leq J(u(t_1))\leq E(t_1)\leq E(0)<d,$$
a contradiction.

For any $t\in[0,T)$, since $I(u(t))<0$, it follows from Lemma \ref{fibering} (iii) that there exists a  $\lambda(t)\in(0,1)$ such that $I(\lambda(t)u(t))=0$, i.e., $\lambda(t)u(t)\in \mathcal{N}$.
By the definition of $d$ and \eqref{i-j}, we have
\begin{equation*}
  \begin{split}
  \dfrac{p-2}{2p}\|u(t)\|_H^2+\dfrac{1}{p^2}\|u(t)\|_p^p\geq& \dfrac{(p-2)\lambda^2(t)}{2p}\|u(t)\|_H^2+\dfrac{\lambda^p(t)}{p^2}\|u(t)\|_p^p\\
  =&J(\lambda^2(t)u(t))\geq d.
  \end{split}
\end{equation*}
The proof is complete.
\end{proof}

With the preliminaries given above, we can show the first blow-up results for problem \eqref{1.1}
with lower initial energy.

\begin{theorem}\label{lower}
Let $p$ satisfy (A), $u_0\in \mathcal{U}$ and $u_1\in L^2(\Omega)$ such that $E(0)<d$.
Then the solution $u(x,t)$ to problem \eqref{1.1} blows up at a finite time $T$ in the sense that
\begin{equation}\label{blow-up}
\lim\limits_{t\rightarrow T^{-}}\Big(\|u(t)\|_2^2+\int_0^t \|u(s)\|^2{\rm d}s\Big)=\infty.
\end{equation}
Moreover, the blow-up time $T$ can be estimated from above as follows
\begin{equation}\label{upper bound}
 T\leq \dfrac{4\Big[\big(a^2+(p-2)^2b\|u_0\|^2_2\big)^{1/2}+a\Big]}{(p-2)^2b},
\end{equation}
where $a,b$ are constants that will be fixed in the proof.
\end{theorem}

\begin{proof}
Assume by contradiction that the solution $u$ exists globally. As was done in \cite{Wu},
fix $T^*>0$ and define the functional
\begin{equation}\label{g}
G(t)=\|u(t)\|_2^2+\int_0^t \|u(s)\|^2{\rm d}s+(T^*-t)\|u_0\|^2+b(t+\tau)^2,\quad t\in[0,T^*],
\end{equation}
where $T^*$, $b$ and $\tau$ are positive constants to be fixed later.
Taking derivative we have
\begin{equation}\label{g1}
G'(t)=2(u,u_t)+2\int_0^t\int_\Omega (uu_s+\nabla u\cdot\nabla u_s){\rm d}x{\rm d}s+2b(t+\tau).
\end{equation}
Taking derivative again and using \eqref{ei}, \eqref{e} and Lemma \ref{inv},
we obtain
\begin{equation}\label{g2}
\begin{split}
G''(t)=&2\|u_t\|_2^2+2(u,u_{tt})+2\int_\Omega (uu_t+\nabla u\cdot\nabla u_t){\rm d}x+2b\\
      =&(p+2)\|u_t\|_2^2+(p-2)\|u\|_H^2-2pE(t)+2b+\frac{2}{p}\|u\|_p^p\\
   \geq&(p+2)\|u_t\|_2^2+2pd-2pE(t)+2b\\
      =&(p+2)\|u_t\|_2^2+2p(d-E(0))+2p\int_0^t\|u_s(s)\|^2{\rm d}s+2b.
\end{split}
\end{equation}
Choosing $b=d-E(0)>0$ and noticing $p>2$ we get
\begin{equation}\label{g3}
G''(t)\geq (p+2)[\|u_t\|_2^2+\int_0^t\|u_s(s)\|^2{\rm d}s+b]>0.
\end{equation}
Combining \eqref{g}, \eqref{g1} with \eqref{g3} we know, for any $t\in[0,T^*]$, that
\begin{equation*}\
\begin{split}
&G(t)G''(t)-\dfrac{p+2}{4}(G'(t))^2\\
\geq&(p+2)\Big[\|u(t)\|_2^2+\int_0^t \|u(s)\|^2{\rm d}s+b(t+\tau)^2\Big]\cdot\Big[\|u_t\|_2^2+\int_0^t\|u_s(s)\|^2{\rm d}s+b\Big]\\
&-(p+2)\Big[(u,u_t)+\int_0^t\int_\Omega (uu_s+\nabla u\cdot\nabla u_s){\rm d}x{\rm d}s+b(t+\tau)\Big]^2.
\end{split}
\end{equation*}
Since
\begin{equation}\label{cauchy-1}
(u,u_t)\leq \|u(t)\|_2\|u_t\|_2,
\end{equation}
\begin{equation}\label{cauchy-2}
\int_0^t\int_\Omega uu_s{\rm d}x{\rm d}s\leq \Big(\int_0^t\int_\Omega u^2{\rm d}x{\rm d}s\Big)^{1/2}\Big(\int_0^t\int_\Omega u^2_s{\rm d}x{\rm d}s\Big)^{1/2},
\end{equation}
\begin{equation}\label{cauchy-3}
\int_0^t\int_\Omega \nabla u\cdot\nabla u_s{\rm d}x{\rm d}s\leq \Big(\int_0^t\int_\Omega |\nabla u|^2{\rm d}x{\rm d}s\Big)^{1/2}\Big(\int_0^t\int_\Omega |\nabla u_s|^2{\rm d}x{\rm d}s\Big)^{1/2},
\end{equation}
it can be directly verified by using Cauchy-Schwarz inequality that
$$G(t)G''(t)-\dfrac{p+2}{4}(G'(t))^2\geq 0,\quad t\in[0,T^*].$$
Take
\begin{equation}\label{tau}
\tau=\max\Big\{0,\frac{2\|u_0\|^2-(p-2)(u_0,u_1)}{(p-2)b}\Big\},
\end{equation}
then
$$G(0)=\|u_0\|_2^2+T^*\|u_0\|^2+b\tau^2>0,$$
$$G'(0)=2(u_0,u_1)+2b\tau>0,$$
and
\begin{equation}\label{T*}
\frac{4G(0)}{(p-2)G'(0)}=\frac{2[\|u_0\|_2^2+T^*\|u_0\|^2+b\tau^2]}{(p-2)[(u_0,u_1)+b\tau]}\leq T^*,
\end{equation}
for suitably large $T^*$.
According to Lemma \ref{concave}, there exists a $T_*>0$ satisfying
\begin{equation}\label{T*}
T_*\leq \frac{4G(0)}{(p-2)G'(0)}
\end{equation}
such that
$$G(t)\rightarrow\infty\ \text{as}\ t\rightarrow T_*^-.$$
This contradicts with the assumption that $G(t)$ is well defined on the closed $[0,T^*]$ for any $T^*>0$.

To derive an upper bound for the blow-up time, we proceed as follows:
Let $T$ be the maximal existence time of $u(x,t)$ (which is finite by the above argument)
and let $G(t)$ be given in \eqref{g}, with the exception that $T^*$ is replaced by $T$ and $t\in[0,\overline{T}]$,
where $\overline{T}\in(0,T)$. Similarly to the foregoing arguments, one can show that
$$\overline{T}\leq \frac{2[\|u_0\|_2^2+T\|u_0\|^2+b\tau^2]}{(p-2)[(u_0,u_1)+b\tau]},$$
where we require that $\tau$, which is independent of $\overline{T}$, still satisfies \eqref{tau}.
It then follows from the arbitrariness of $\overline{T}<T$ that
\begin{equation}\label{t-1}
T\leq \frac{2[\|u_0\|_2^2+T\|u_0\|^2+b\tau^2]}{(p-2)[(u_0,u_1)+b\tau]},
\end{equation}
which guarantees
\begin{equation}\label{t-2}
 T\leq T(\tau)\triangleq\frac{2(\|u_0\|^2_2+b\tau^2)}{(p-2)[(u_0,u_1)+b\tau]-2\|u_0\|^2}.
\end{equation}
Set $a=2\|u_0\|^2-(p-2)(u_0,u_1)$ and $\tau_0=\dfrac{[a^2+(p-2)^2b\|u_0\|_2^2]^{1/2}+a}{(p-2)b}$.
Then it is an easy matter to verify that $\tau_0$ satisfies \eqref{tau}, $T(\tau)$ attains its minimum at $\tau_0$
and
$$T(\tau_0)=\dfrac{4\Big[\big(a^2+(p-2)^2b\|u_0\|^2_2\big)^{1/2}+a\Big]}{(p-2)^2b}.$$
Therefore,
$$T\leq \dfrac{4\Big[\big(a^2+(p-2)^2b\|u_0\|^2_2\big)^{1/2}+a\Big]}{(p-2)^2b}.$$
The proof is complete.
\end{proof}

\begin{remark}\label{rem3.1}
  By \eqref{e} and \eqref{i-j} and recalling that $p>2$ one sees $E(0)<0$ implies $I(u_0)<0$.
  Therefore, Theorem \ref{lower} implies that the solution $u(x,t)$ to problem \eqref{1.1} blows up in finite time
  for negative initial energy.
\end{remark}

\section{Blow-up for high initial energy}
\setcounter{equation}{0}

In this section we shall build a blow-up criterion for problem \eqref{1.1} at high initial energy level.
Some ideas used in this section are borrowed from \cite{Guo-Li} and \cite{Li-Liu}.
As a preliminary, we first establish a lemma that will play a fundamental role.

\begin{lemma}\label{lem1}
Let $p$ satisfy $(A)$. Assume that $u_0\in H$ and $u_1\in L^2(\Omega)$
such that
\begin{equation}\label{initial}
0<E(0)<\dfrac{C_0}{p}(u_0,u_1).
\end{equation}
Then the solution $u(x,t)$ to problem \eqref{1.1} satisfies
\begin{equation}\label{growth}
(u,u_t)-\frac{p}{C_0}E(t)\geq \Big[(u_0,u_1)-\dfrac{p}{C_0}E(0)\Big]e^{C_0t},\quad t\in[0,T).
\end{equation}
Here
\begin{equation}\label{c0}
C_0=\min\Big\{p+2,p(p-2)\lambda_1,\frac{(p-2)(\lambda_1+\lambda^2_1)}{2}\Big\}>0.
\end{equation}
\end{lemma}

\begin{proof}
Set $F(t)=(u,u_t)$. By direct calculations and recalling \eqref{e} we have
\begin{equation}\label{4.3}
\begin{split}
   F'(t)=&\|u_t\|_2^2+(u,u_{tt}) \\
     =&\|u_t\|_2^2+(u,-\Delta^2u+\Delta u+\Delta u_t-u_t+|u|^{p-2}u\ln|u|)\\
     =&\|u_t\|_2^2-\|u\|_H^2-\int_\Omega \nabla u\cdot\nabla u_t{\rm d}x-(u,u_t)+\int_\Omega |u|^p\ln|u|{\rm d}x\\
     =&\dfrac{p+2}{2}\|u_t\|_2^2+\dfrac{p-2}{2}\|u\|_H^2-\int_\Omega \nabla u\cdot\nabla u_t{\rm d}x-(u,u_t)-pE(t)+\dfrac{1}{p}\|u\|_p^p\\
     \geq&\dfrac{p+2}{2}\|u_t\|_2^2+\dfrac{p-2}{2}\|u\|_H^2-\int_\Omega \nabla u\cdot\nabla u_t{\rm d}x-(u,u_t)-pE(t).
\end{split}
\end{equation}
By using Cauchy inequality, we can estimate the third and fourth terms in the last inequality as follows
\begin{equation}\label{4.4}
 |\int_\Omega \nabla u\cdot\nabla u_t{\rm d}x|\leq\dfrac{C_0}{4p}\|\nabla u\|_2^2+\dfrac{p}{C_0}\|\nabla u_t\|_2^2,
\end{equation}
\begin{equation}\label{4.5}
 |(u,u_t)|\leq\dfrac{C_0}{4p}\|u\|_2^2+\dfrac{p}{C_0}\|u_t\|_2^2.
\end{equation}
Substituting \eqref{4.4} and \eqref{4.5} into \eqref{4.3} we arrive at
\begin{equation}\label{4.6}
 F'(t)\geq \dfrac{p+2}{2}\|u_t\|_2^2+\dfrac{p-2}{2}\|u\|_H^2-\dfrac{C_0}{4p}\|u\|^2-\dfrac{p}{C_0}\| u_t\|^2-pE(t).
\end{equation}

Set $H(t)=F(t)-\dfrac{p}{C_0}E(t)$.
Then in view of \eqref{lam}, \eqref{ei}, \eqref{c0} and \eqref{4.6} we obtain
\begin{equation}\label{4.7}
\begin{split}
   H'(t)=&F'(t)-\dfrac{p}{C_0}E'(t)=F'(t)+\dfrac{p}{C_0}\|u_t\|^2\\
     \geq&\dfrac{p+2}{2}\|u_t\|_2^2+\dfrac{p-2}{2}\|u\|_H^2-\dfrac{C_0}{4p}\|u\|^2-pE(t)\\
     \geq&\dfrac{p+2}{2}\|u_t\|_2^2+\Big[\dfrac{(p-2)\lambda_1}{2}-\dfrac{p(p-2)\lambda_1}{4p}\Big]\|u\|^2-pE(t)\\
     =&\dfrac{p+2}{2}\|u_t\|_2^2+\dfrac{(p-2)\lambda_1}{4}\|u\|^2-pE(t)\\
     \geq&\dfrac{p+2}{2}\|u_t\|_2^2+\dfrac{(p-2)(\lambda_1+\lambda^2_1)}{4}\|u\|_2^2-pE(t)\\
     \geq&C_0\Big[\dfrac{1}{2}\|u_t\|_2^2+\dfrac{1}{2}\|u\|_2^2-\dfrac{p}{C_0}E(t)\Big]\geq C_0H(t).
\end{split}
\end{equation}
Since $H(0)>0$ by \eqref{initial}, \eqref{growth} follows after an application of Gronwall's inequality to $H(t)$.
The proof is complete.
\end{proof}

With Lemma \ref{lem1} at hand, we are now in the position to
prove high initial energy blow-up and estimate an upper bound
for the blow-up time for problem \eqref{1.1}.

\begin{theorem}\label{high}
Let all the assumptions in Lemma \ref{lem1} hold.
Then the solution $u(x,t)$ to problem \eqref{1.1} blows up at some finite time $T$ in the sense of \eqref{blow-up}.
Moreover, if
\begin{equation}\label{c2}
  E(0)<\frac{C_0}{p}\|u_0\|_2^2,
\end{equation}
then
\begin{equation}\label{t}
  T\leq \dfrac{2(\|u_0\|_2^2+\beta t_0^2)}{(p-2)[(u_0,u_1)+\beta t_0]-2\|u_0\|^2}.
\end{equation}
Here $C_0$ is the positive constant given in \eqref{c0},
\begin{equation}\label{beta}
\beta=2\Big[\frac{C_0}{p}\|u_0\|_2^2-E(0)\Big]>0,
\end{equation}
$t_0$ is suitably large such that
\begin{equation}\label{t0}
(p-2)\Big[(u_0,u_1)+\beta t_0\Big]>2\|u_0\|^2.
\end{equation}
\end{theorem}

\begin{proof}
We divide the proof into two steps.

{\bf Step I: Finite time blow-up.}
Suppose by contradiction that \eqref{blow-up} will not happen for any finite $T$.
Then $\|u(\cdot,t)\|_2$ is well-defined for all $t\geq0$.
Without loss of generality, we may assume that $E(t)\geq0$ for all $t\geq0$.
Otherwise by Remark \ref{rem3.1} we know that $u(x,t)$ blows up in finite time.

On one hand, it follows from Lemma \ref{lem1} that
\begin{equation}\label{4.13}
\dfrac{d}{dt}\|u(t)\|_2^2=2(u,u_t)\geq 2H(0)e^{C_0t}+\dfrac{2p}{C_0}E(t)\geq 2H(0)e^{C_0t}.
\end{equation}
Integration of \eqref{4.13} over $[0,t]$ yields
\begin{equation}\label{4.14}
\begin{split}
\|u(t)\|_2^2=&\|u_0\|_2^2+2\int_0^t\int_\Omega uu_\tau {\rm d}x{\rm d}\tau\geq\|u_0\|_2^2+2\int_0^t H(0)e^{C_0\tau}{\rm d}\tau\\
=&\|u_0\|_2^2+\dfrac{2H(0)}{C_0}(e^{C_0t}-1).
\end{split}
\end{equation}

On the other hand, by virtue of Minkowski inequality,
H\"{o}lder inequality, \eqref{ei}, the definition of $\lambda_1$ and the fact $E(t)\geq0$ one gets
\begin{equation*}
\begin{split}
\|u(t)\|_2\leq&\|u_0\|_2+\|u(t)-u_0\|_2=\|u_0\|_2+\|\int_0^t u_\tau{\rm d}\tau\|_2\\
\leq&\|u_0\|_2+\int_0^t \|u_\tau\|_2{\rm d}\tau\leq \|u_0\|_2
+\dfrac{1}{\sqrt{1+\lambda_1}}\int_0^t \|u_\tau\|{\rm d}\tau\\
\leq&\|u_0\|_2
+\dfrac{\sqrt{t}}{\sqrt{1+\lambda_1}}\Big(\int_0^t \|u_\tau\|^2{\rm d}\tau\Big)^{1/2}=\|u_0\|_2+\dfrac{\sqrt{t}}{\sqrt{1+\lambda_1}}(E(0)-E(t))^{1/2}\\
\leq&\|u_0\|_2+\sqrt{\dfrac{E(0)}{1+\lambda_1}}t^{1/2},
\end{split}
\end{equation*}
which contradicts \eqref{4.14} when $t$ is sufficiently large.
Therefore, $u(x,t)$ blows up in finite time.

{\bf Step II: Upper bound for the blow-up time.} From now on, we assume that $T>0$ is the blow-up time of $u(x,t)$,
which is finite by Step I.
According to Lemma \ref{lem1} and the assumption that $E(t)\geq 0$, we see that
\begin{equation*}
\dfrac{d}{dt}\|u(t)\|_2^2=2(u,u_t)\geq 2H(0)e^{C_0t}+\dfrac{2p}{C_0}E(t)> 0,\quad t\in[0,T),
\end{equation*}
which implies $\|u(t)\|_2^2$ is increasing with respect to $t$.
To estimate $T$ from above, as was done in the proof of Theorem \ref{lower}, we define
\begin{equation}\label{g}
K(t)=\|u(t)\|_2^2+\int_0^t \|u(s)\|^2{\rm d}s+(T-t)\|u_0\|^2+\beta(t+t_0)^2,\quad \quad t\in[0,T),
\end{equation}
where $\beta$ and $t_0$ are given in \eqref{beta} and \eqref{t0}, respectively.
By applying similar argument to that in the proof of the first part of Theorem \ref{lower}, we obtain
\begin{equation}\label{concave-inequality}
\begin{split}
&K(t)K''(t)-\dfrac{p+2}{4}(K'(t))^2\\
=&2K(t)\Big[\|u_t\|_2^2-\|u\|_H^2+\int_\Omega |u|^p\ln|u|{\rm d}x+\beta\Big]\\
&-(p+2)\Big[(u,u_t)+\int_0^t\int_\Omega (uu_s+\nabla u\cdot\nabla u_s){\rm d}x{\rm d}s+\beta(t+t_0)\Big]^2\\
=&2K(t)\Big[\|u_t\|_2^2-\|u\|_H^2+\int_\Omega |u|^p\ln|u|{\rm d}x+\beta\Big]\\
&+(p-2)\Big[\eta(t)-\big(K(t)-(T-t)\|u_0\|^2\big)\big(\|u_t\|_2^2+\int_0^t \|u_s\|^2{\rm d}s+\beta\big)\Big],
\end{split}
\end{equation}
where
\begin{equation}\label{concave-inequality-1}
\begin{split}
\eta(t)=&\Big[\|u(t)\|_2^2+\int_0^t \|u(s)\|^2{\rm d}s+\beta(t+t_0)^2\Big]\Big[\|u_t\|_2^2+\int_0^t \|u_s\|^2{\rm d}s+\beta\Big]\\
&-\Big[(u,u_t)+\int_0^t\int_\Omega (uu_s+\nabla u\cdot\nabla u_s){\rm d}x{\rm d}s+\beta(t+t_0)\Big]^2.
\end{split}
\end{equation}
Using \eqref{cauchy-1}-\eqref{cauchy-3} and Cauchy-Schwarz inequality we can show that $\eta(t)\geq 0$ on $[0,T)$.
Therefore, recalling \eqref{ei}, \eqref{e}, \eqref{beta} and the monotonicity of $\|u(t)\|_2^2$, we have
\begin{equation}\label{concave-inequality-2}
\begin{split}
&K(t)K''(t)-\dfrac{p+2}{4}(K'(t))^2\\
\geq &2K(t)\Big[\|u_t\|_2^2-\|u\|_H^2+\int_\Omega |u|^p\ln|u|{\rm d}x+\beta\Big]-(p+2)K(t)\Big[\|u_t\|_2^2+\int_0^t \|u_s\|^2{\rm d}s+\beta\Big]\\
=&K(t)\Big[(p-2)\|u\|_H^2-2pE(0)+(p-2)\int_0^t \|u_s\|^2{\rm d}s+\frac{2}{p}\|u\|_p^p-p\beta\Big]\\
\geq&K(t)\Big[(p-2)(\lambda_1+\lambda^2_1)\|u\|_2^2-2pE(0)-p\beta\Big]\\
\geq&K(t)\Big[(p-2)(\lambda_1+\lambda^2_1)\|u_0\|_2^2-2pE(0)-p\beta\Big]\\
=&2pK(t)\Big[\frac{C_0}{p}\|u_0\|_2^2-E(0)-\beta/2\Big]\geq0,\quad t\in[0,T).
\end{split}
\end{equation}
Besides, $K(0)=\|u_0\|_2^2+T\|u_0\|^2+\beta t_0^2>0$ and $K'(0)=2(u_0,u_1)+2\beta t_0>0$ by \eqref{t0}.
Applying Lemma \ref{concave} to $K(t)$ yields
\begin{equation*}
T\leq\dfrac{4K(0)}{(p-2)K'(0)}=\dfrac{2(\|u_0\|_2^2+T\|u_0\|^2+\beta t_0^2)}{(p-2)[(u_0,u_1)+\beta t_0]}.
\end{equation*}
Since $\dfrac{2\|u_0\|^2}{(p-2)[(u_0,u_1)+\beta t_0]}<1$ by \eqref{t0}, we further obtain
\begin{equation*}
T\leq\dfrac{2(\|u_0\|_2^2+\beta t_0^2)}{(p-2)[(u_0,u_1)+\beta t_0]-2\|u_0\|^2}.
\end{equation*}
The proof of Theorem \ref{high} is complete.
\end{proof}

\begin{remark}
As was done in deriving \eqref{upper bound}, one can also minimize the right-hand side term of \eqref{t} for $t_0$ satisfying \eqref{t0}
to obtain a more accurate upper bound for $T$. Interested reader may check it.
\end{remark}

\section{Lower bound for the blow-up time}
\setcounter{equation}{0}

Since the lower bound for the blow-up time provides a safe time interval for the system under consideration,
it is more important in practice to estimate $T$ from below. In this section,
our aim is to determine a lower bound for the blow-up time of problem \eqref{1.1} by constructing a new auxiliary functional.
Throughout this section we shall use $C,C_1,C_2,\cdots,$ to denote generic positive constants
which may depend on $\Omega,p,n$, but are independent of the solution $u(x,t)$.

\begin{theorem}\label{lower-bound}
Assume that $p$ satisfies $\dfrac{2n}{n+2}<\dfrac{2n(p-1)}{n+2}<2_*$, i.e.,
\begin{equation}\label{p}
 p\in(2,\infty)\quad if\quad  n\leq 4;\quad p\in\big(2,\frac{2n-2}{n-4}\big)\quad  if \quad n\geq 5.
\end{equation}
Let $u(x,t)$ be a weak solution to problem \eqref{1.1} that blows up at $T$ in the sense of \eqref{blow-up}.
Then $$T\geq \int_{N(0)}^\infty\dfrac{{\rm d}s}{C_4+C_5s^{p-1+\mu}},$$
where $N(0)=\|u_1\|_2^2+\|u_0\|_H^2$.
\end{theorem}

\begin{proof}
For simplicity, we only prove this theorem for $n\geq 3$. The case for $n=1,2$ is similar (and simpler).
We aim to determine a time interval $(0,T_0)$ on which the quantity $\|u(t)\|_H^2$ is bounded.
Clearly $T_0$ is a lower bound for $T$ since both $\|u(t)\|_2^2$ and $\|u(t)\|^2$ can be bounded by $\|u(t)\|_H^2$.

Define
\begin{equation}\label{n1}
N(t)=\|u_t(t)\|_2^2+\|u(t)\|_H^2,\quad t\in[0,T_0).
\end{equation}
Then
\begin{equation}\label{n2}
\lim_{t\rightarrow T_0^{-}}N(t)=+\infty.
\end{equation}
Differentiating \eqref{n1} and making use of Green's second identity,
we obtain
\begin{align}\label{n3}
N'(t)&=2[(u_t,u_{tt})+(\Delta u,\Delta u_t)+(\nabla u,\nabla u_t)]\nonumber\\
&=2(u_t,u_{tt}+\Delta^2 u-\Delta u)\nonumber\\
&=2(u_t,\Delta u_t-u_t+|u|^{p-2}\ln|u|)\nonumber\\
&=-2\|u_t\|^2+2\int_\Omega u_t|u|^{p-2}u\ln|u|{\rm d}x.
\end{align}

Set
$\Omega_1=\Omega_1(t)=\{x\in\Omega: |u(x,t)|<1\}$ and $\Omega_2=\Omega_2(t)=\{x\in\Omega: |u(x,t)|\geq1\}$.
Since $p$ satisfies \eqref{p}, we can choose $\mu>0$ suitably small such that $\frac{2n(p-1+\mu)}{n+2}<2_*$,
which implies that $H$ can be embedded into $L^{\frac{2n(p-1+\mu)}{n+2}}(\Omega)$ continuously.
We use $B_\mu$ to denote the embedding constant from $H$ to $L^{\frac{2n(p-1+\mu)}{n+2}}(\Omega)$,
i.e.,
\begin{equation}\label{n4}
\|v\|_{\frac{2n(p-1+\mu)}{n+2}}\leq B_\mu\|v\|_H,\qquad\forall\ v\in H.
\end{equation}
Using H\"{o}lder's inequality, Cauchy inequality, \eqref{n4} and the basic inequalities $|s^{p-2}\ln s|\leq (e(p-1))^{-1}$ for $0<s<1$ and $\ln s\leq\dfrac{1}{e\mu}s^\mu$
for $s\geq 1$, we can estimate the second term on the right-hand side of \eqref{n3} as follows
\begin{align}\label{n5}
&\int_\Omega u_t|u|^{p-2}u\ln|u|{\rm d}x\nonumber=\int_{\Omega_1} u_t|u|^{p-2}u\ln|u|{\rm d}x+\int_{\Omega_2} u_t|u|^{p-2}u\ln|u|{\rm d}x\nonumber\\
\leq&\Big(\int_{\Omega_1} |u_t|^{\frac{2n}{n-2}}{\rm d}x\Big)^{\frac{n-2}{2n}}\Big(\int_{\Omega_1} \big||u|^{p-2}u\ln|u|\big|^{\frac{2n}{n+2}}{\rm d}x\Big)^{\frac{n+2}{2n}}\nonumber\\
&+\Big(\int_{\Omega_2} |u_t|^{\frac{2n}{n-2}}{\rm d}x\Big)^{\frac{n-2}{2n}}\Big(\int_{\Omega_2} \big||u|^{p-2}u\ln|u|\big|^{\frac{2n}{n+2}}{\rm d}x\Big)^{\frac{n+2}{2n}}\nonumber\\
\leq&\|u_t\|_{\frac{2n}{n-2}}\Big[(e(p-1))^{-1}|\Omega_1|^{\frac{n+2}{2n}}+(e\mu)^{-1}\Big(\int_{\Omega} |u|^{\frac{2n(p-1+\mu)}{n+2}}{\rm d}x\Big)^{\frac{n+2}{2n}}\Big]\nonumber\\
\leq&C\|u_t\|\Big[(e(p-1))^{-1}|\Omega_1|^{\frac{n+2}{2n}}+(e\mu)^{-1}B^{p-1+\mu}_\mu\|u\|^{p-1+\mu}_H\Big]\nonumber\\
\leq&\varepsilon \|u_t\|^2+C(\varepsilon)\Big[C_1+C_2\|u\|^{2(p-1+\mu)}_H\Big]\nonumber\\
\leq&\varepsilon \|u_t\|^2+C(\varepsilon)\Big[C_1+C_3N^{p-1+\mu}(t)\Big].
\end{align}
Therefore, it follows by taking $\varepsilon\leq 1$ and substituting \eqref{n5} into \eqref{n3} that
\begin{equation}\label{n6}
N'(t)\leq C_4+C_5N^{p-1+\mu}(t).
\end{equation}
Integrating \eqref{n6} over $[0,t]$, we have
\begin{equation}\label{n7}
\int_0^t\dfrac{N'(\tau)}{C_4+C_5N^{p-1+\mu}(\tau)}{\rm d}\tau\leq t.
\end{equation}
Letting $t\rightarrow T_0^{-}$ and recalling \eqref{n2}, we obtain
\begin{equation}\label{n8}
\int_{N(0)}^\infty\dfrac{{\rm d}s}{C_4+C_5s^{p-1+\mu}}\leq T_0\leq T.
\end{equation}
Recalling that $p-1+\mu>1$, the left-hand side term in \eqref{n8} is finite.
The proof is complete.
\end{proof}

\begin{remark}
  By making full use of the damping term, we obtain the lower bound for the blow-up time not only for subcritical exponent $p$, but also for some supercritical ones. We point out that this observation can also be applied to
  problem \eqref{sec-2} considered in \cite{Di}.
\end{remark}

{\bf Acknowledgement}\\
The authors would like to express their sincere gratitude to Professor Wenjie Gao for his enthusiastic
guidance and constant encouragement.

\end{document}